\documentclass{article}
\usepackage{amsmath}
\usepackage{amsthm}
\usepackage{amssymb}
\usepackage{graphicx}
\usepackage{graphics}
\usepackage{listings}
\usepackage{multirow}
\usepackage{framed}
\usepackage{hyperref}
\usepackage{caption}

\usepackage{algorithm}
\usepackage{algorithmic}

\usepackage{tikz}
\usepackage{pgf}
\usetikzlibrary{arrows}
\usepackage{mathtools}
\usepackage{mathrsfs}

\usepackage{color}
\definecolor{red}{rgb}{1,0,0}

\definecolor{blue}{rgb}{0,0,1}

\newcommand{\ms}{\medskip}
\newcommand{\bpf}{\begin{proof}}
\newcommand{\epf}{\end{proof}\ms}

\newtheorem{theorem}{Theorem}
\newtheorem{corollary}{Corollary}
\newtheorem{lemma}{Lemma}
\newtheorem{proposition}{Proposition}

\newtheorem{observation}{Observation}

\newtheorem{example}{Example}

\theoremstyle{definition}
\newtheorem{definition}{Definition}

\begin{document}
\date{}
\title{Connected power domination in graphs}


\author{Boris Brimkov\thanks{Department of Computational and Applied Mathematics, Rice University, Houston, TX, 77005, USA (boris.brimkov@rice.edu, derek.j.mikesell@rice.edu, logan.smith@rice.edu)}, 
Derek Mikesell$^*$,
Logan Smith$^*$}

\maketitle

\begin{abstract}
The study of power domination in graphs arises from the problem of placing a minimum number of measurement devices in an electrical network while monitoring the entire network. A power dominating set of a graph is a set of vertices from which every vertex in the graph can be observed, following a set of rules for power system monitoring. In this paper, we study the problem of finding a minimum power dominating set which is connected; the cardinality of such a set is called the \emph{connected power domination number} of the graph. We show that the connected power domination number of a graph is NP-hard to compute in general, but can be computed in linear time in cactus graphs and block graphs. We also give various structural results about connected power domination, including a cut vertex decomposition and a characterization of the effects of various vertex and edge operations on the connected power domination number. Finally, we present novel integer programming formulations for power domination, connected power domination, and power propagation time, and give computational results.
\end{abstract}

\section{Introduction}

Electrical power companies must constantly monitor their electrical networks in order to detect and respond to failures in the networks. To this end, they place devices called Phase Measurement Units (PMUs) at select locations in the system. A PMU can directly measure the currents and phase angles of all transmission lines incident to its location. Moreover, physical laws governing electrical circuits (e.g. Kirchhoff's circuit laws) can be leveraged to gain information about parts of the network which are not directly observed. Due to the high cost of PMUs, it is a problem of interest to find the smallest number of PMUs (and their locations) from which an entire network can be observed. This PMU placement problem has been explored extensively in the electrical engineering literature; see \cite{EEprobabilistic,Baldwin93,Brunei93,EEinformationTheoretic,EEtaxonomy,Mili91,EEtabuSearch,EEmultiStage}, and the bibliographies therein for various placement strategies and computational results. 

Haynes et al. \cite{powerdom3} formulated the PMU placement problem as a dynamic graph coloring problem, where vertices represent electric nodes and edges represent connections via transmission lines. In this model, a set of initially colored vertices in a graph (corresponding to locations of PMUs) cause other vertices to become colored (i.e. to be observed by the PMUs); the goal is to find the smallest set of initially colored vertices which causes all other vertices to become colored. More precisely, let $G=(V,E)$ be a graph and let $S\subset V$ be a set of initially colored vertices; the physical laws by which PMUs can observe a network give rise to the following color change rules (see \cite{BH05}):
\begin{enumerate}
\item[1)] Every neighbor of an initially colored vertex becomes colored.
\item[2)] Whenever there is a colored vertex with exactly one uncolored neighbor, that neighbor becomes colored.
\end{enumerate}
$S$ is a {\em power dominating set} of $G$ if all vertices in $G$ become colored after applying rule 1) once, and rule 2) as many times as possible (i.e. until no more vertices can change color). The {\em power domination number} of $G$, denoted $\gamma_P(G)$, is the cardinality of a minimum power dominating set. $S$ is a {\em zero forcing set} of $G$ if all vertices in $G$ become colored after applying rule 2) as many times as possible (and not applying rule 1) at all). The {\em zero forcing number} of $G$, denoted $Z(G)$, is the cardinality of a minimum zero forcing set. The process of zero forcing was introduced independently in combinatorial matrix theory \cite{AIM-Workshop} and in quantum control theory \cite{quantum1}. 

In this paper, we study a variant of power domination which requires every set of initially colored vertices to induce a connected subgraph. Given a connected graph $G=(V,E)$, a set $S\subset V$ is a \emph{connected power dominating set} of $G$ if $S$ is a power dominating set and $G[S]$ is connected. The \emph{connected power domination number}, denoted $\gamma_{P,c}(G)$, is the cardinality of a minimum connected power dominating set. Requiring a power dominating set to be connected is motivated by the application in monitoring electrical networks: the data from PMUs is relayed by high-speed communication infrastructure to processing stations which collect and manage this data; thus, in addition to minimizing the production costs of the PMUs, an electric power company may seek to place all PMUs in a compact, connected region in the network in order to reduce the number of processing stations and related infrastructure required to collect the data.

Connected power domination was explored from a computational perspective in \cite{FanWatson} (although the problem called ``connected power domination" in \cite{FanWatson} is slightly different from the one considered here; see Section 6 for details). The connected variants of other graph problems, including connected zero forcing \cite{brimkov_fast,brimkov_extremal,brimkov_comp_and_complex}, connected domination \cite{Caro,Desormeaux,connected_dom,Sampathkumar}, and connected vertex cover \cite{con_vc2,con_vc1}, have also been extensively studied. Imposing connectivity often fundamentally changes the nature of a problem, including its complexity, structural properties, and applications. Other generalizations and extensions of power domination have also been explored \cite{chang2012,chang2015}, as has the problem of studying the number of timesteps in which the graph is colored by a power dominating set \cite{aazami2,ferrero17,liao16}.

The paper is organized as follows. In the next section, we recall some graph theoretic notions and notation. In Section 3, we present novel structural results about connected power domination, as well as some technical lemmas which are used in the sequel. In Section 4, we prove that connected power domination is NP-complete. In Section 5, we give efficient algorithms for the connected power domination numbers of trees, block graphs, and cactus graphs. In Section 6, we provide integer programming models for power domination, connected power domination, and power propagation time, and give computational results. We conclude with some final remarks and open questions in Section 7.

\section{Preliminaries}

A graph $G=(V,E)$ consists of a vertex set $V$ and an edge set $E$ of two-element subsets of $V$. The \emph{order} and \emph{size} of $G$ are denoted by $n=|V|$ and $m=|E|$, respectively. Two vertices $v,w\in V$ are \emph{adjacent}, or \emph{neighbors}, if $\{v,w\}\in E$. If $v$ is adjacent to $w$, we write $v\sim w$; otherwise, we write $v\not\sim w$. The \emph{neighborhood} of $v\in V$ is the set of all vertices which are adjacent to $v$, denoted $N(v;G)$; the \emph{degree} of $v\in V$ is defined as $d(v;G)=|N(v;G)|$. The \textit{closed neighborhood} of $v \in V$ is the set $N(v;G)\cup\{v\}$, denoted $N[v;G]$. The dependence of these parameters on $G$ can be omitted when it is clear from the context. The \textit{closed neighborhood} of a set $S \subset V$ is the set $N[S]=\cup_{v\in S}N[v]$. A set $S$ in $G$ is said to \textit{dominate} $N[S]$. Given $S \subset V$, the \emph{induced subgraph} $G[S]$ is the subgraph of $G$ whose vertex set is $S$ and whose edge set consists of all edges of $G$ which have both endpoints in $S$. The number of connected components of $G$ will be denoted by $c(G)$; an isomorphism between graphs $G_1$ and $G_2$ will be denoted by $G_1\simeq G_2$. A \emph{leaf}, or \emph{pendant}, is a vertex with degree 1. A \emph{cut vertex} is a vertex which, when removed, increases the number of connected components in $G$. A \emph{cut edge} is an edge which, when removed, increases the number of components of $G$. A \emph{biconnected component}, or \emph{block}, of $G$ is a maximal subgraph of $G$ which has no cut vertices. The \emph{disjoint union} of sets $A$ and $B$ will be denoted $A \dot\cup B$.

A \emph{chronological list of forces} $\mathcal{F}$ associated with a power dominating or zero forcing set $S$ of a graph $G$ is a sequence of forces applied to color $V(G)$ in the order they are applied. A \emph{forcing chain} for a chronological list of forces is a maximal sequence of vertices  $(v_1,\ldots,v_k)$ such that the force $v_i\to v_{i+1}$ is in $\mathcal{F}$ for $1\leq i\leq k-1$. Each forcing chain produces a distinct path in $G$, one of whose endpoints is in $S$; we will say this endpoint \emph{initiates} the forcing chain.

We also recall some terminology and notation from \cite{brimkov_comp_and_complex} which will be used in the sequel. Let $G=(V,E)\not\simeq P_n$ be a graph and $v$ be a vertex of degree at least $3$. A \emph{pendant path attached to} $v$ is a maximal set $P\subset V$ such that $G[P]$ is a connected component of $G-v$ which is a path, one of whose ends is adjacent to $v$ in $G$. The neighbor of $v$ in $P$ will be called the \emph{base} of the path, and $p(v)$ will denote the number of pendant paths attached to $v\in V$. We will also say that $p(u)=1$ if $u$ is a cut vertex which belongs to a pendant path. Similarly, a \emph{pendant tree attached to} $v$ is a set $T\subset V$ composed of the vertices of a connected component of $G-v$ which is a tree, and which has a single vertex adjacent to $v$ in $G$. 
Finally, for a connected graph $G=(V,E)\not\simeq P_n$, define:
\begin{eqnarray*}
R_1(G)&=&\{v\in V: c(G-v)=2, \; p(v)=1\}\\
R_2(G)&=&\{v\in V: c(G-v)=2, \; p(v)=0\}\\
R_3(G)&=&\{v\in V: c(G-v)\geq 3\}\\
\mathcal{M}(G)&=&R_2(G)\cup R_3(G). 
\end{eqnarray*} 
When there is no scope for confusion, the dependence on $G$ will be omitted. Note that the sets $R_1$, $R_2$, and $R_3$ partition the set of cut vertices of $G$. For convenience, we will say that $\mathcal{M}(P_n)=\emptyset$. For other graph theoretic terminology and definitions, we refer the reader to \cite{bondy}.

\section{Structural results}

We first present two technical lemmas about vertices which are contained in every connected power dominating set, and vertices which are not contained in any minimum connected power dominating set.

\begin{lemma}
\label{MR_lemma}
Let $G=(V,E)$ be a connected graph different from a path and $R$ be an arbitrary connected power dominating set of $G$. Then $\mathcal{M}(G)\subset R$.
\end{lemma}

\begin{proof}
Let $v$ be a cut vertex of $G$ which is in $R_3(G)$. If two or more components of $G-v$ contain vertices of $R$, then since $R$ is connected and since any path between two vertices from different components of $G-v$ also contains $v$, $R$ must contain $v$. Now suppose all vertices of $R$ are contained in a single component $G_1$ of $G-v$. If $v$ is not contained in $R$, then no vertex outside $G_1\cup\{v\}$ can be dominated or forced at any timestep since $v$ will have at least two uncolored neighbors in the other components of $G-v$; this is a contradiction. Thus, every vertex of $R_3$ must be in $R$. 

Now let $v$ be a cut vertex of $G$ which is in $R_2(G)$. If both components of $G-v$ contain vertices of $R$, then by the same argument as above, $R$ must contain $v$. Now suppose that all vertices of $R$ are contained in a single component $G_1$ of $G-v$. If $v$ is not contained in $R$, then since the other component of $G-v$ is not a pendant path of $G$, it cannot be forced by a single forcing chain passing through $v$, a contradiction. Thus, every vertex of $R_2$ must be in $R$. 
Since $\mathcal{M}=R_2\cup R_3$, it follows that $\mathcal{M}\subset R$. 
\end{proof}

\begin{lemma}
\label{no_leaf_lemma}
Let $G$ be a graph different from a path. Then, no minimum connected power dominating set of $G$ contains a leaf of $G$.
\end{lemma}
\begin{proof}
Suppose for contradiction that there is a minimum connected power dominating set $S$ of $G$ which contains a leaf $\ell$ of $G$. Since $G$ is not a path, $S$ must contain another vertex $u$ besides $\ell$. Since $S$ is connected, it must include the neighbor $v$ of $\ell$, since every path between $u$ and $\ell$ passes through $v$. However, $S\backslash \{\ell\}$ is also a connected power dominating set of $G$, since $v$ can dominate $\ell$ in the first timestep, and removing $\ell$ does not disconnect $G[S]$. This contradicts the minimality of $S$.
\end{proof}

\subsection{Cut vertex decomposition}

A \emph{trivial block} is a block consisting of two vertices, both of which belong to the same pendant path. A \emph{nontrivial} block is a block $B$ which is not trivial, together with all pendant paths attached to vertices in $B$. In this section, we present a technique for computing the connected power domination numbers of graphs with cut vertices in terms of the connected power domination numbers of their nontrivial blocks. We first introduce several definitions and lemmas which will be used in the main result of this section; some of these are adapted from \cite{restrictedPD,powerdom3}.

\begin{definition}\label{restrictPD}
Let $G=(V,E)$ be a graph and let $X\subseteq V$. A set $S\subseteq V(G)$ is a {\em connected power dominating set of $G$ subject to $X$} if $S$ is a connected power dominating set of $G$ and $X\subseteq S$. The size of a minimum connected power dominating set subject to $X$ is denoted $\gamma_{P,c}(G;X)$.
\end{definition}

\begin{example}
Let $P_n$ be a path with vertices $v_1,\ldots,v_n$, where $v_i\sim v_{i+1}$ for $1\leq i\leq n-1$. Then, for any $1\leq i,j\leq n$, we have $\gamma_{P,c}(P_n;\{v_i,v_j\})=|i-j|+1$.
\end{example}

\begin{observation}
By Lemma \ref{MR_lemma}, for any graph $G$, $\gamma_{P,c}(G;\mathcal{M}(G))=\gamma_{P,c}(G)$.
\end{observation}

\noindent The following observation is well-known in the power domination literature. It is often stated in terms of leaves, but remains valid for pendant paths.

\begin{observation}
\label{obs_leaves}
For any graph $G$, there exists a minimum power dominating set of $G$ that contains every vertex that is incident to two or more leaves.
\end{observation}

\noindent Given a graph $G=(V,E)$ and a set $X\subseteq V$, define $\ell _r(G,X)$ as the graph obtained by attaching $r$ leaves to each vertex in $X$. The following result establishes a relationship between $\gamma_{P,c}(G;X)$ and $\ell_3(G,X)$.

\begin{proposition}
\label{prop_leaf_graph}
For any graph $G=(V,E)$ and any set $X\subseteq V$, a set $S$ is a minimum connected power dominating set of $G$ subject to $X$ if and only if $S$ is a minimum connected power dominating set of $\ell_3(G,X)$.
\end{proposition}
\begin{proof}
Let $S$ be a minimum connected power dominating set of $G$ subject to $X$. Then, $S$ can dominate all the added leaves in $\ell_3(G,X)$ in the first timestep, and then force the rest of the vertices of $V(G)$ as in $G$. Hence $S$ is a connected power dominating set of $\ell_3(G,X)$. Suppose there is a minimum connected power dominating set $S'$ of $\ell_3(G,X)$ with $|S'|<|S|$. By Lemma~\ref{MR_lemma}, $S'$ contains all vertices in $X$. Moreover, since $S'$ is minimum, it does not contain any of the added leaves of $\ell_3(G,X)$. Thus, $S'$ is a connected power dominating set of $G$ subject to $X$, a contradiction to $S$ being minimum. The other direction follows analogously. 
\end{proof}

\noindent Let $\mu(v)$ denote the number of nontrivial blocks a vertex $v$ belongs to.

\begin{theorem}
\label{theorem_blocks}
Let $G$ be a graph with nontrivial blocks $B_1,\ldots,B_k$. For $1\leq i\leq k$, let $A_i=V(B_i)\cap \mathcal{M}(G)$. Then,
\begin{equation}
\label{eq_split}
\gamma_{P,c}(G)=\sum_{i=1}^k\gamma_{P,c}(\ell_3(B_i,A_i))-\sum_{v\in \mathcal{M}(G)}(\mu(v)-1)
\end{equation}
\end{theorem}

\begin{proof}
Let $S_i$ be an arbitrary minimum connected power dominating set of $B_i$ subject to $A_i$. By Proposition \ref{prop_leaf_graph}, $S_i$ is also a minimum connected power dominating set of $\ell_3(B_i,A_i)$. We claim that $S:=\cup_{i=1}^k S_i$ is a minimum connected power dominating set of $G$. 

Let $u$ and $v$ be two arbitrary vertices in $S$. Without loss of generality, suppose $u$ and $v$ are not in $\mathcal{M}(G)$; the case when one or both of $u$ and $v$ are in $\mathcal{M}(G)$ is handled analogously. If $u$ and $v$ belong to the same nontrivial block $B_i$, then since $S$ contains $S_i$, and $S_i$ is a connected power dominating set containing $u$ and $v$, there is a path between $u$ and $v$ in $G[S]$. Otherwise, let $B_{i_0}$ and $B_{i_p}$ respectively be the blocks containing $u$ and $v$. Let $B_{i_0},a_{i_1},B_{i_1},\ldots,a_{i_{p-1}},B_{i_{p-1}},a_{i_p},B_{i_p}$ be the path between $B_{i_0}$ and $B_{i_p}$ in the block tree of $G$, where $a_{i_1},\ldots,a_{i_p}$ are the cut vertices; note that by definition, each of these vertices is in $\mathcal{M}(G)$. Then since $S$ contains $S_{i_0}$, and $S_{i_0}$ is a connected power dominating set containing $u$ and $a_{i_1}$, there is a path between $u$ and $a_{i_1}$ in $G[S]$. Likewise, there is a path between $a_{i_j}$ and $a_{i_{j+1}}$ for $1\leq j\leq p-1$, and there is a path between $a_{i_p}$ and $v$. Thus, in all cases, there is a path between $u$ and $v$ in $G[S]$, so $S$ is a connected set. 

Moreover, $S$ is a power dominating set, since by construction each $S_i$ can dominate the corresponding nontrivial block of $G$ (because no uncolored vertex within a nontrivial block has uncolored neighbors outside that block).

Finally, it remains to be shown that $S$ is minimum. Suppose $S'$ is a connected power dominating set of $G$ with $|S'|<|S|$. Thus, there is some $i$ such that $|S'\cap B_i|<|S\cap B_i|=|S_i|$. However, by Lemma \ref{MR_lemma}, $S'$ contains all vertices in $\mathcal{M}(G)$, and hence $S'\cap B_i$ must contain $A_i$; this contradicts the minimality of $S_i$. It follows that $S$ is a minimum connected power dominating set of $G$. Thus, $S$ is a minimum connected power dominating set of $G$. Moreover,

\begin{eqnarray*}
|S|&=&\left|\bigcup_{i=1}^k S_i\right|=\left|\bigcup_{i=1}^k ((S_i\backslash A_i)\cup A_i)\right|=\left|\bigcup_{i=1}^k (S_i\backslash A_i)\cup \bigcup_{i=1}^k A_i\right|=\\
&=&\sum_{i=1}^k|S_i\backslash A_i| + \left|\bigcup_{i=1}^k A_i\right|=\sum_{i=1}^k|S_i|-\sum_{i=1}^k|A_i| + \left|\bigcup_{i=1}^k A_i\right|=\\
&=&\sum_{i=1}^k\gamma_{P,c}(\ell_3(B_i,A_i))-\sum_{v\in \mathcal{M}(G)}|\mu(v)| + |\mathcal{M}(G)|.
\end{eqnarray*}
Thus, $\gamma_{P,c}(G)=|S|$ is as claimed in (\ref{eq_split}).
\end{proof}

\subsection{Vertex and edge operations}

In this section, we explore the effects of various vertex and edge operations on the connected power domination number. The effects of operations such as deletion and contraction of vertices and edges have been studied for other parameters as well. For example, the \textit{power domination spread} of a vertex $v$ and edge $e$ is defined as $\gamma_P(G;v) = \gamma_P(G) - \gamma_P (G - v)$ and $\gamma_P(G;e) = \gamma_P(G) - \gamma_P (G - e)$, respectively. 
The analogously defined \textit{rank spread} $r$, \textit{zero forcing spread} $z$, \emph{path spread} $p$, and \emph{connected forcing spread} $z_c$ have also been defined as parameters which respectively describe the change in the minimum rank, zero forcing number, path cover number, and connected forcing number of a graph when a vertex $v$ or edge $e$ is deleted. In particular, it has been shown that: 

\vspace{-8pt}
\begin{center}
\noindent\begin{minipage}{.35\linewidth}
\begin{eqnarray*}
-\infty<&\hspace{-7pt}\gamma_P(G;v)&\hspace{-7pt} \leq 1\\
-1 \leq &\hspace{-7pt}z(G;v)&\hspace{-7pt}  \leq 1\\
0 \leq &\hspace{-7pt} r(G;v)&\hspace{-7pt}  \leq 2\\
-1 \leq &\hspace{-7pt} p(G;v)&\hspace{-7pt}  \leq 1\\
-\infty<&\hspace{-7pt} z_c(G;v)&\hspace{-7pt} <\infty
\end{eqnarray*}
\end{minipage}%
\begin{minipage}{.18\linewidth}
\begin{eqnarray*}
-1\leq &\hspace{-7pt}\gamma_P(G;e)&\hspace{-7pt}\leq 1\\
-1 \leq &\hspace{-7pt}z(G;e)&\hspace{-7pt}  \leq 1\\
-1 \leq &\hspace{-7pt} r(G;e)&\hspace{-7pt}  \leq 1\\
-1 \leq &\hspace{-7pt} p(G;e)&\hspace{-7pt}  \leq 1\\
-\infty<&\hspace{-7pt} z_c(G;e)&\hspace{-7pt} <\infty
\end{eqnarray*}
\end{minipage}
\begin{minipage}{.2\linewidth}
\begin{eqnarray*}
&&\text{\cite{benson, chang2012}},\\
&&\text{\cite{Edholm,Huang}},\\
&&\text{\cite{barioli_minrank,nylen}},\\
&&\text{\cite{barioli_minrank,
barioli_pathcover,row_cacti}},\\
&&\text{\cite{brimkov_comp_and_complex}}.
\end{eqnarray*}
\end{minipage}
\end{center}
\vspace{8pt}

We now define the \textit{connected power domination spread} of a vertex $v$ and edge $e$ as $\gamma_{P,c}(G;v) = \gamma_{P,c}(G) - \gamma_{P,c} (G - v)$ and $\gamma_{P,c}(G;e) = \gamma_{P,c}(G) - \gamma_{P,c} (G - e)$, respectively. Since a disconnected graph cannot have a connected power dominating set, in the definitions above we will restrict $v$ to be a non-cut vertex and $e$ to be a non-cut edge. We now show that the connected power domination spread of a vertex or edge can be arbitrarily large.

\begin{proposition} \label{prop:edge_sub}
For any integer $c > 0$, there exist graphs $G_1$, $G_2$, $G_3$, and $G_4$, vertices $v_1\in G_1$ and $v_2 \in G_2$, and edges $e_3\in G_3$ and $e_4 \in G_4$, such that $\gamma_{P,c}(G_1;v_1) = -c$, $\gamma_{P,c}(G_2;v_2) = c$, $\gamma_{P,c}(G_3;e_3) = -c$, and $\gamma_{P,c}(G_4;e_4) = c$.

\end{proposition}
\begin{proof}
Let $P=(V,E)$ be a path on $c+3$ vertices, with end-vertices $v_1$ and $v_n$. Let $v_2$ and $v_{n-1}$ be the neighbors of $v_1$ and $v_n$, respectively. Let $G=(V\dot\cup\{x,y,z\}, E\dot\cup \{xv_1,xv_2,xy,yz,zv_{n-1}\})$; see Figure \ref{figure_op_constructions}, left, for an illustration. 

It can be verified that $\gamma_{P,c}(G)=1$, $\gamma_{P,c}(G-y)=c+1$, and $\gamma_{P,c}(G-y-x)=1$. Thus, $\gamma_{P,c}(G;y)=-c$ and $\gamma_{P,c}(G-y;x)=c$. Similarly, it can be verified that $\gamma_{P,c}(G-xy)=c+1$, and $\gamma_{P,c}(G-xy-xv_2)=1$. Thus, $\gamma_{P,c}(G;xy)=-c$ and $\gamma_{P,c}(G-xy;xv_2)=c$.
\end{proof}

From Proposition \ref{prop:edge_sub}, it follows that the connected power domination number can be arbitrarily increased or decreased with the addition or deletion of a single vertex or edge. Furthermore, the connected power domination number is not monotone with respect to either operation. We next show that the connected power domination number can be arbitrarily increased or decreased with the contraction of an edge, and it can be arbitrarily increased, but not decreased, with the subdivision of an edge. Given a graph $G$ and edge $e$ in $G$, let $G:e$ denote the graph obtained by subdividing $e$ in $G$.

\begin{figure}[ht!]
\begin{center}
\includegraphics[scale=.3]{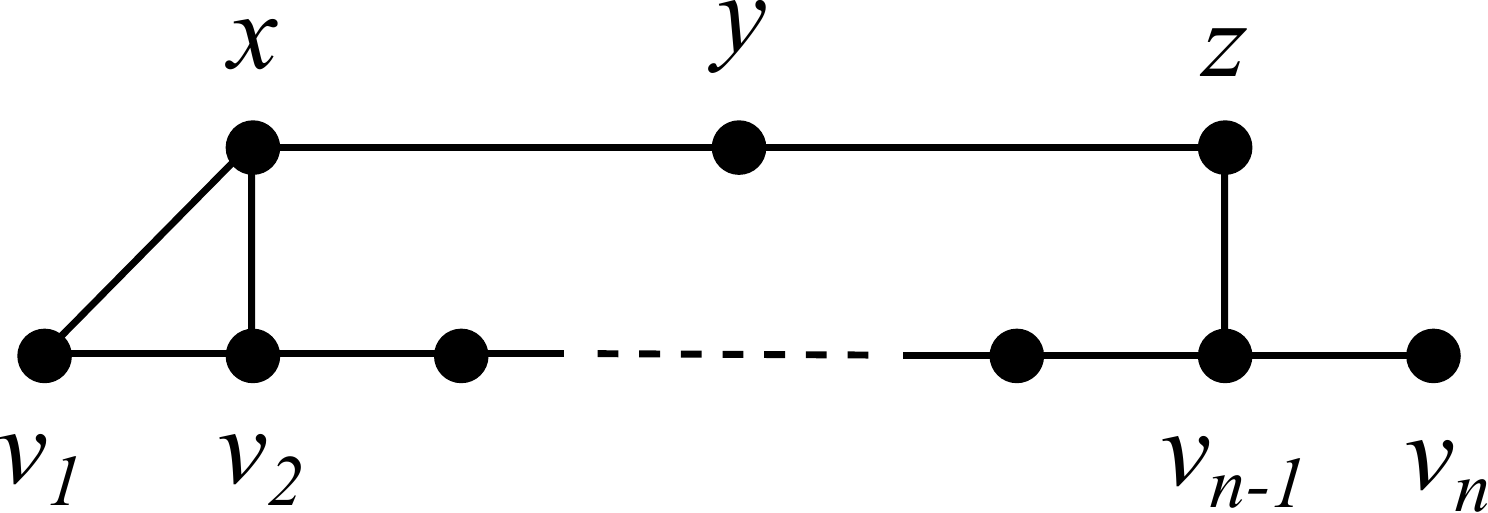}\qquad\qquad
\includegraphics[scale=.3]{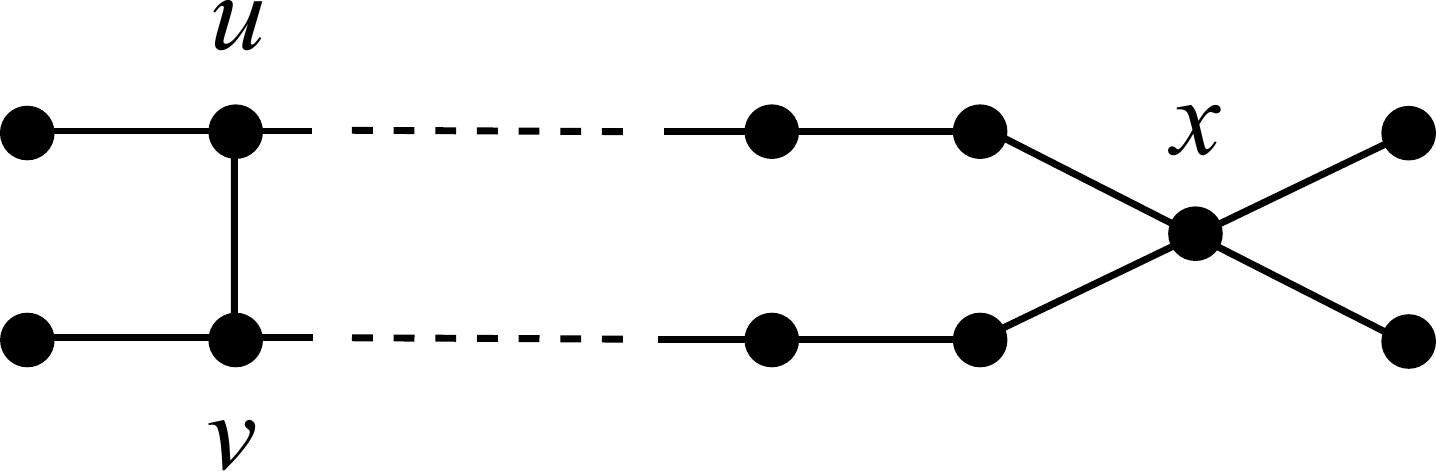}
\caption{Deleting a vertex or edge, and contracting an edge, can change the connected power domination number arbitrarily. Subdividing an edge can arbitrarily increase, but not decrease, the connected power domination number.}
\label{figure_op_constructions}
\end{center}
\end{figure}

\begin{proposition}
For any integer $c>0$, there exist graphs $G_1$, $G_2$, and $G_3$ and edges $e_1\in G_1$, $e_2\in G_2$, $e_3 \in G_3$, such that $\gamma_{P,c}(G_1:e_1) - \gamma_{P,c}(G_1) = c$, $\gamma_{P,c}(G_2) - \gamma_{P,c}(G_2/e_2) = -c$, and $\gamma_{P,c}(G_3) - \gamma_{P,c}(G_3/e_3) = c$.
\end{proposition}
\begin{proof}
Let $C=(V,E)$ be a cycle on $2c+1$ vertices, let $uv$ be an edge of $C$, and let $x$ be a vertex of $C$ at maximum distance from $u$ and $v$. Let $G=(V\dot\cup\{\ell_1,\ell_2,\ell_3,\ell_4\},E\dot\cup\{\ell_1x,\ell_2x,\ell_3u,\ell_4v\})$; see Figure \ref{figure_op_constructions}, right, for an illustration. 

It can be verified that $\gamma_{P,c}(G)=1$, $\gamma_{P,c}(G:uv)=c+1$, and $\gamma_{P,c}(G/uv)=c+1$. Thus, $\gamma_{P,c}(G:uv) - \gamma_{P,c}(G) = c$, $\gamma_{P,c}(G) -\gamma_{P,c}(G/uv) = -c$, and $\gamma_{P,c}(G:uv)-\gamma_{P,c}((G:uv)/uz)=c$, where $z$ is the new vertex introduced after subdividing edge $uv$ in $G:uv$.
\end{proof}

\begin{proposition}
Let $G$ be a connected graph and $uv$ be an edge in $G$. Then $\gamma_{P,c}(G:uv) \geq \gamma_{P,c}(G)$, and this bound is tight.
\end{proposition}
\begin{proof}
Let $w$ be the new vertex introduced after subdividing edge $uv$ in $G:uv$. Let $R$ be a minimum connected power dominating set of $G:uv$ and fix a chronological list of forces associated with $R$. We will show that $R$ or $R\backslash \{w\}$ is also a connected power dominating set of $G$. 

Suppose first that neither $u$ nor $v$ is contained in $R$. Since $R$ is connected, $w$ cannot be in $R$; thus, $R$ remains connected in $G$. Moreover, in $G:uv$, $w$ is either forced by $u$ or $v$, say $u$, and $v$ is forced either by $w$ or by some other vertex. The same chronological list of forces remains valid in $G$, except that if $w$ forces $v$ in $G:uv$, $u$ forces $v$ in $G$. Thus, $R$ is a connected power dominating set of $G$. Next, suppose that exactly one of $u$ and $v$, say $u$, is contained in $R$. Then $R$ remains connected in $G$, $u$ can dominate $v$ in $G$, and all other forces associated with $R$ in $G:uv$ remain valid in $G$. Thus, $R$ is a connected power dominating set in $G$. Finally, suppose that both $u$ and $v$ are contained in $R$. If $w$ is also contained in $R$, then $R\backslash \{w\}$ is a connected set in $G$, $u$ can dominate $v$ in $G$, and all other forces associated with $R$ in $G:uv$ remain valid in $G$. Similarly, if $w$ is not contained in $R$, then $R$ remains connected in $G$, and all forces associated with $R$ in $G:uv$ remain valid in $G$. It follows that $\gamma_{P,c}(G)\leq \gamma_{P,c}(G:uv)$. This bound holds with equality, e.g., when any edge of a path $P_n$ is subdivided.
\end{proof}

\section{NP-completeness of connected power domination}

In this section, we show that computing the connected power domination number of a graph is NP-complete. To begin, we state the decision version of this problem.\\

\noindent PROBLEM: Connected power domination ($CPD$)\\
INSTANCE: A simple undirected connected graph $G=(V,E)$ and a positive integer $k\leq |V|$.\\
QUESTION: Does $G$ contain a power dominating set $S$ of size at most $k$ such that $G[S]$ is connected?

\begin{theorem}
\label{np_theorem}
$CPD$ is NP-complete.
\end{theorem}
\begin{proof}
Given a graph $G=(V,E)$ and a set $S\subset V$, clearly it can be verified in polynomial time that $S$ is a power dominating set, that $G[S]$ is connected, and that $|S|\leq k$. Thus, $CPD$ is in NP.

For our reduction, we will use the problem of zero forcing, which was proved to be NP-complete in \cite{aazami,fast_mixed_search}. The decision version of zero forcing is stated below. \\

\noindent PROBLEM: Zero forcing ($ZF$)\\
INSTANCE: A simple undirected graph $G=(V,E)$ and a positive integer $k\leq |V|$.\\
QUESTION: Does $G$ contain a zero forcing set $S$ of size at most $k$?\\

Next, we construct a transformation $f$ from $ZF$ to $CPD$.
Let $I=\langle G,k \rangle$ be an instance of $ZF$, where $G=(V,E)$ and $V=\{v_1,\ldots,v_n\}$. We define $f(I)=\langle G',k+1\rangle$, where 
\begin{eqnarray*}
G'&=&(V\cup\{v^*,\ell_1,\ell_2\}\cup\{u_i,w_i:1\leq i\leq n\}\cup\{p_{i,j},q_{i,j}:1\leq i,j\leq n\},\\
&&E\cup\{\{v_i,u_i\},\{p_{i,1},u_i\},\{q_{i,1},u_i\},\{p_{i,n},w_i\},\{p_{i,n},v^*\},\{q_{i,n},v^*\}:1\leq i\leq n\}\cup\\
&&\{\{p_{i,j},p_{i,j+1}\},\{q_{i,j},q_{i,j+1}\}:1\leq j\leq n-1, 1\leq i\leq n\}\cup\{\{v^*,\ell_1\},\{v^*,\ell_2\}\}).
\end{eqnarray*}
Clearly, $G'$ can be constructed from $G$ in polynomial time, so $f$ is a polynomial transformation. See Figure \ref{fig_cf_complexity} for an illustration of $G$ and $G'$.

\begin{figure}[ht!]
\begin{center}
\includegraphics[scale=0.35]{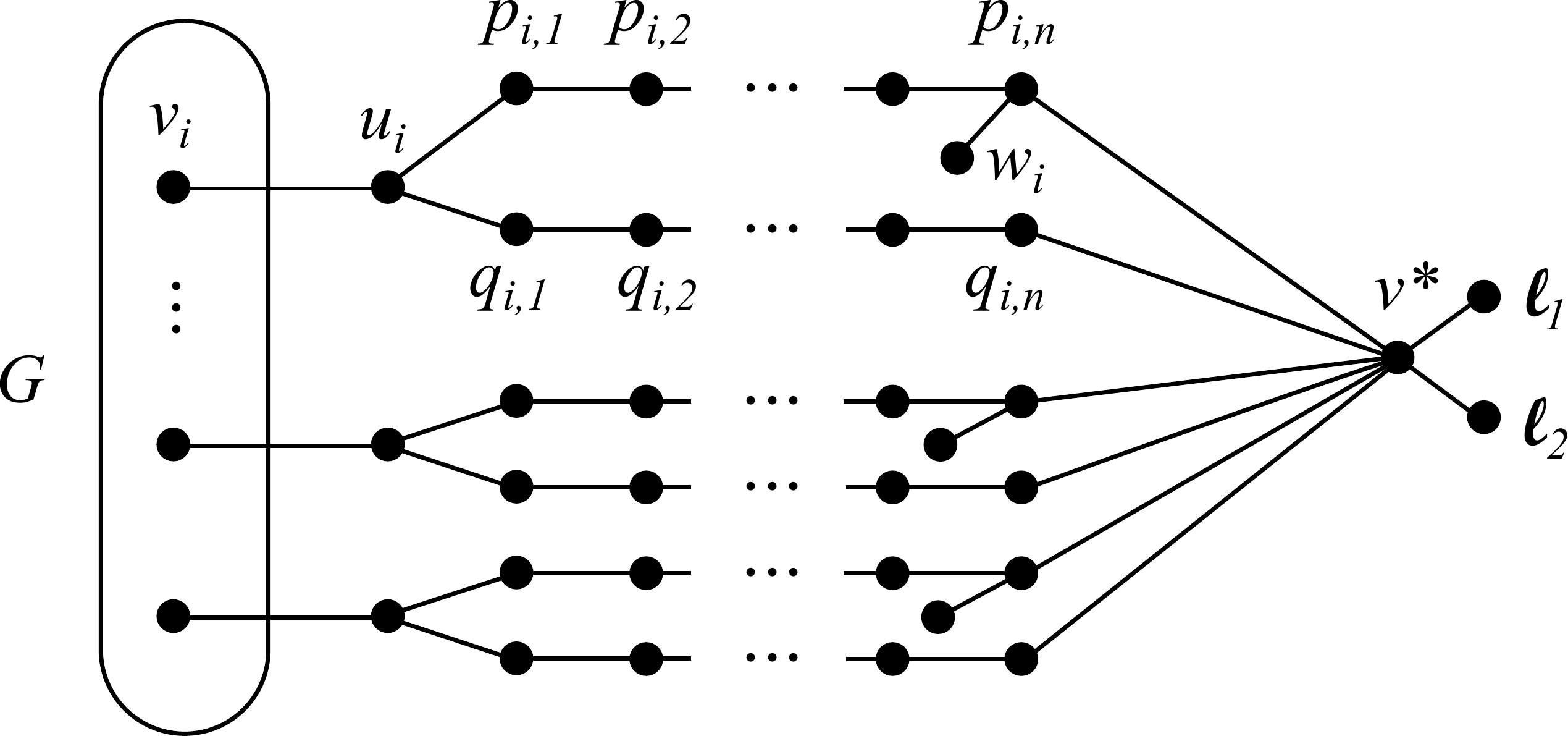}
\caption{Obtaining $G'$ from $G$.}
\label{fig_cf_complexity}
\end{center}
\end{figure}

We will now prove the correctness of $f$. Suppose $I=\langle G,k\rangle$ is a \emph{`yes'} instance of $ZF$, i.e., that $G=(V,E)$ has a zero forcing set $S=\{v_i:i\in J\}$ where $J\subset \{1,\ldots,n\}$ is some index set of size at most $k$. We claim that $S':=\{p_{i,n}:i\in J\} \cup \{v^*\}$ is a connected power dominating set of $G'$. To see why, first note that since $v^*$ is adjacent to every vertex in $S'\backslash\{v^*\}$, $G'[S']$ is connected. Next, note that $v^*$ can dominate all vertices in $\{p_{i,n},q_{i,n}:1\leq i\leq n\}\cup \{\ell_1,\ell_2\}$ at the first timestep, and the vertices in $\{p_{i,n}:i\in J\}$ can dominate the vertices in $\{p_{i,n-1}:i\in J\}\cup \{w_i:i\in J\}$ in the first timestep. Then, the vertices in $\{q_{i,j},u_i: 1\leq j\leq n-1, 1\leq i\leq n\}$ can be colored by forcing chains initiated at $q_{i,n}$, $1\leq i\leq n$. Moreover, the vertices in $\{p_{i,j}: 1\leq j\leq n-1, i\in J\}$ can be colored by forcing chains initiated at $p_{i,n}$, $i\in J$. When all these vertices are colored, the vertices in $\{u_i:i\in J\}$ will each have a single uncolored neighbor, and hence will be able to force the vertices in $\{v_i:i\in J\}$. At that point, since $\{v_i:i\in J\}$ is a zero forcing set of $G$ and since all neighbors of vertices in $V(G)$ are colored, all vertices in $V(G)$ can get forced. Finally, all vertices in $\{p_{i,j},w_i:1\leq i\leq n, i\notin J,1\leq j\leq n-1\}$ will get colored by forcing chains initiated at $u_i$, $1\leq i\leq n$, $i\notin J$. Thus $S'$ is a connected power dominating set of $G'$ of size at most $k+1$, so $f(I)=\langle G',k+1\rangle$ is a \emph{`yes'} instance of $CPD$.

Conversely, suppose $f(I)=\langle G',k+1\rangle$ is a \emph{`yes'} instance of $CPD$, i.e., that $G'$ has a connected power dominating set of size at most $k+1$. Let $S'$ be a minimum connected power dominating set of $G'$.
By Lemma \ref{MR_lemma}, $v^*$ is in $S'$. By Lemma \ref{no_leaf_lemma}, no leaf of $G'$ is in $S'$; thus, $\ell_1$, $\ell_2$, and $w_i$, $1\leq i\leq n$, are not in $S'$. No vertex of $V(G)\cup \{u_i:1\leq i\leq n\}$ can be in $S'$, since the shortest path between $v^*$ and any of these vertices passes through more than $n+1$ vertices, $v^*$ is in $S'$, $S'$ is connected, and $|S'|\leq k+1\leq n+1$. If $S'$ contains a vertex $q_{i^*,j^*}$ for some $1\leq i^*,j^*\leq n$, then it must also contain all vertices in $\{q_{i^*,j}:j^*<j\leq n\}$. Note that $q_{i^*,j^*}$ can dominate or force only vertices in the path $G'[\{q_{i^*,j}:1\leq j\leq n\}\cup\{u_{i^*}\}]$. However, since $v^*$ dominates $q_{i^*,n}$ and $q_{i^*,n}$ can initiate a forcing chain which colors all vertices in $\{q_{i^*,j}:1\leq j\leq n\}\cup\{u_{i^*}\}$, it follows that $S'\backslash \{q_{i^*,j}:j^*\leq j\leq n\}$ is also a connected power dominating set of $G'$. This contradicts the minimality of $S'$, so $S'$ does not contain any vertices of $\{q_{i,j}:1\leq i,j\leq n\}$. Similarly, $S'$ does not contain any vertices of $\{p_{i,j}:1\leq i\leq n,1\leq j\leq n-1\}$. Thus, $S'$ consists of $v^*$ and some vertices of $\{p_{i,n}:1\leq i\leq n\}$. Let $J$ be the index set of these vertices, i.e., $S'=\{v^*\}\cup \{p_{i,n}:i\in J\}$. 

By the same argument as above, the vertices in $S'$ dominate or force the vertices in $\{\ell_1,\ell_2\}\cup\{q_{i,j}:1\leq i,j\leq n\}\cup\{u_i:1\leq i\leq n\}\cup\{p_{i,n}:1\leq i\leq n\}\cup\{p_{i,j}:1\leq j\leq n-1, i\in J\}\cup\{w_i:i\in J\}\cup \{v_i:i\in J\}$. At the timestep when all these vertices get colored, no vertex in $V(G')\backslash V(G)$ can perform a force; the only way a vertex $p_{i,j}$, $1\leq i\leq n, i\notin J,1\leq j\leq n-1$ can get colored is for $v_i$ to get colored and then for $u_i$ to initiate a forcing chain containing $p_{i,j}$. Moreover, the only way a vertex in $V(G)$ can get colored is if it is forced by some other vertex in $V(G)$. It follows that $\{v_i:i\in J\}$ must be a zero forcing set of $G$, since otherwise some $v_i\in V(G)$ will never be colored, contradicting that $S'$ is a connected power dominating set. Thus, if $f(I)$ is a \emph{`yes'} instance of $CPD$, then $I$ is a \emph{`yes'} instance of $ZF$.
\end{proof}

\section{Characterizations of $\gamma_{P,c}$ for specific graphs}

While the connected power domination number is NP-hard to compute in general, in this section we show that the connected power domination numbers of cactus and block graphs can be computed efficiently.

\subsection{Trees}

We begin with a closed formula for the connected power domination number of trees, and characterize the trees whose connected power domination number equals their power domination number.

\begin{theorem}
\label{tree_thm}
Let $T=(V,E)$ be a tree. Then, $\gamma_{P,c}(T)=\max\{1,|\mathcal{M}(T)|\}$. Moreover, a minimum connected power dominating set of $T$ can be found in $O(n)$ time.
\end{theorem}

\begin{proof}
By definition, $\mathcal{M}(T)=\emptyset$ if and only if $T$ is a path; in this case, $\gamma_{P,c}(T)=1$ and any vertex of the path is a minimum connected power dominating set. Assume henceforth that $T$ is not a path; we will show that $\mathcal{M}$ is a minimum connected power dominating set of $T$. All vertices of $T$ with degree at least 3 are in $R_3$, and all vertices of $T$ which have degree 2 and do not belong to pendant paths are in $R_2$. Thus, any vertex of $T$ which is not initially colored belongs to some pendant path, and all other vertices of that pendant path are also initially uncolored. Since deleting all vertices of a pendant path does not disconnect the graph, $\mathcal{M}$ is a connected set. Next, let $v$ be a vertex to which a pendant path is attached. Since $v$ is in $R_3$ and therefore in $\mathcal{M}$, $v$ can dominate all its neighbors in the first timestep; then, the base of each pendant path can initiate a forcing chain which forces the whole path.
Thus, $\mathcal{M}$ is a connected power dominating set of $T$, so $\gamma_{P,c}(T)\leq |\mathcal{M}|$. Moreover, by Lemma \ref{MR_lemma}, every minimum connected power dominating set of $T$ contains $\mathcal{M}$, so $\gamma_{P,c}(T)\geq |\mathcal{M}|$, and hence $\gamma_{P,c}(T)=|\mathcal{M}|$.

The vertices in $R_1(T)$ can be found in $O(n)$ time (e.g. by starting from the degree 1 vertices of the graph and applying depth-first-search until a vertex of degree at least 3 is reached). Thus, the set $\mathcal{M}(T)=V(T)\backslash R_1(T)$ can also be found in $O(n)$ time. 
\end{proof}

\noindent Since the set $\mathcal{M}(G)$ is uniquely determined, we have the following corollary to Theorem \ref{tree_thm}. 

\begin{corollary}
If $T$ is a tree different from a path, $T$ has a unique minimum connected power dominating set. 
\end{corollary}

\noindent The next result gives a characterization of trees for which $\gamma_{P,c}(T)=\gamma_P(T)$. 

%
%

\begin{proposition}
\label{prop_tree_equality}
Let $T = (V,E)$ be a tree. Then $\gamma_{P,c}(T) = \gamma_{P}(T)$ if and only if $T$ satisfies one of the following conditions:
\begin{enumerate}
\item[$1)$] $T$ is a path.
\item[$2)$] $T$ is a tree such that for every vertex $v\in V(T)$, if $d(v)=2$, $v$ belongs to a pendant path and if $d(v)\geq 3$, $v$ is the base of at least 2 pendant paths.
\end{enumerate}
\end{proposition}
\begin{proof}
It is easy to see if $T$ is a path, $\gamma_{P,c}(T) = \gamma_{P}(T)$; thus, we will assume henceforth that $T$ is not a path. 

Let $T$ be a tree such that $\gamma_{P,c}(T)=\gamma_{P}(T)$, and let $S$ be a minimum connected power dominating set of $T$. If $T$ has a vertex $v$ with $d(v)=2$ which is not in a pendant path, $v$ must be in $R_2$ and therefore in $S$. However, $S \backslash \{v\}$ is clearly also a (non-connected) power dominating set since any neighbor of $v$ can force $v$, so $\gamma_P(T) < \gamma_{P,c}(T)$. Similarly, if $T$ has a vertex $v$ with $d(v)\geq 3$ which is adjacent to 1 or 0 pendant paths, we claim that $S\backslash \{v\}$ is a power dominating set. Indeed, since $v$ has at least 3 neighbors, at most one of which belongs to a pendant path, $v$ has a neighbor in $R_2$ or $R_3$, and can be dominated by that neighbor. Moreover, the pendant path attached to $v$ (if it exists) can be forced by $v$ in later timesteps since all other neighbors of $v$ are colored. Thus, it follows that $\gamma_P(T)<\gamma_{P,c}(T)$.

Now suppose $T$ is a (non-path) tree which satisfies condition 2) in Proposition \ref{prop_tree_equality}. Since all degree 2 vertices of $T$ belong to pendant paths, $R_2(T)=\emptyset$. Thus, by Theorem \ref{tree_thm}, $\mathcal{M}(T)=R_3(T)$ is a minimum connected power dominating set of $T$. Moreover, since every vertex of degree at least 3 has at least two pendant paths attached to it, by Observation \ref{obs_leaves}, there exists a minimum power dominating set which contains every vertex in $R_3$. Thus, $\gamma_P(T)\geq |R_3|=\gamma_{P,c}(T)\geq \gamma_P(T)$, so $\gamma_{P,c}(T)=\gamma_P(T)$.
\end{proof}

\subsection{Block graphs}

We now extend the result of Theorem \ref{tree_thm} to block graphs. A \emph{block graph} is a graph whose biconnected components are cliques.

\begin{theorem}
\label{block_graph_thm}
Let $G=(V,E)$ be a block graph. Then, $\gamma_{P,c}(G)=\max\{1,|\mathcal{M}(G)|\}$.
\end{theorem}
\begin{proof}
If $\mathcal{M}(G)=\emptyset$, then either $G\simeq P_n$, or $G\not\simeq P_n$ and all cut vertices of $G$ are in $R_1(G)$. In the latter case, $G$ consists of a single clique $K$ of size at least 3, and at most one pendant path attached to each vertex of $K$. Then, a single vertex of $K$ is a (connected) power dominating set of $G$, since it can dominate all vertices of $K$ in the first timestep, and then any pendant paths attached to vertices of $K$ can be forced by their respective bases. Thus, if $\mathcal{M}(G)=\emptyset$, $\gamma_{P,c}(G)=1$.

Now suppose that $\mathcal{M}(G)\neq\emptyset$; we claim that $\mathcal{M}(G)$ is a connected power dominating set of $G$. Note that in a block graph, the shortest path between two vertices is unique; in particular, the shortest path between two cut vertices in $\mathcal{M}(G)$ is entirely composed of other cut vertices in $\mathcal{M}(G)$. Thus, $\mathcal{M}(G)$ is a connected set. If $v$ belongs to a maximal clique $K$ of $G$ of size 2 which is part of a pendant tree $T$, then $v$ is either in $\mathcal{M}(G)$, or is in a pendant path and gets forced by $\mathcal{M}(G)\cap T$ by Theorem \ref{tree_thm}. If $v$ belongs to a maximal clique $K$ of $G$ of size 2 which is not part of a pendant tree, then $v$ is in $R_2(G)$ (and hence in $\mathcal{M}(G)$) and is therefore initially colored. Now suppose $v$ belongs to a clique $K$ of $G$ of size at least 3. Since $\mathcal{M}(G)\neq \emptyset$, every clique of $G$ of size at least 3 contains at least one vertex of $\mathcal{M}(G)$. Thus, a vertex in $\mathcal{M}(G)\cap K$ can dominate $v$ in the first timestep. Since $v$ gets colored in every case, $\mathcal{M}(G)$ is a connected power dominating set of $G$. By Lemma $\ref{MR_lemma}$, every connected power dominating set of $G$ contains $\mathcal{M}(G)$, so $\mathcal{M}(G)$ is a minimum connected power dominating set of $G$.
\end{proof}

\subsection{Cactus graphs}

In this section, we give a linear-time algorithm for finding a minimum connected power dominating set of a cactus graph. A \emph{cactus graph} is a graph in which any two cycles have at most one common vertex; equivalently, it is a graph whose biconnected components are cycles or cut edges. We first establish some results which are applicable to arbitrary graphs containing a cycle block. 

Let $G$ be a connected graph and $C$ be the vertex set of a block of $G$ such that $G[C]$ is a cycle. Given vertices $u$ and $v$ of $C$, let $(u\hookrightarrow v)$ be the set of vertices of $C$ (given a plane embedding) encountered while traveling counterclockwise from $u$ to $v$, not including $u$ and $v$. We will refer to $(u\hookrightarrow v)$ as a \emph{segment} of $C$. Note that a segment of the form $(u\hookrightarrow u)$ is also well-defined.


\begin{observation}
\label{one_seg_obs}
Let $G=(V,E)$ be a connected graph and $C$ be the vertex set of a block of $G$ such that $G[C]$ is a cycle. Then, any set $R\subset V$ such that $G[R]$ is connected can exclude at most one segment of $C$.
\end{observation}

In particular, a connected power dominating set of a graph $G$ with cycle block $C$ can exclude at most one segment of $C$.

\begin{lemma}
\label{two_ap_lemma}
Let $G=(V,E)$ be a connected graph and $C$ be the vertex set of a block of $G$ such that $G[C]$ is a cycle. A segment of $C$ can be excluded from a connected power dominating set of $G$ if and only if it contains:
\begin{itemize}
\item[--] 0 cut vertices of $C$, or
\item[--] 1 cut vertex of $C$ which is in $R_1(G)$, or 
\item[--] 2 cut vertices of $C$, which are adjacent and both of which are in $R_1(G)$.
\end{itemize}
\end{lemma}
\begin{proof}
We will first show that if a segment of $C$ is excluded from a connected power dominating set of $G$, then it contains either no cut vertices, or one cut vertex in $R_1(G)$, or two adjacent cut vertices in $R_1(G)$. Let $R$ be an arbitrary connected power dominating set of $G$ and $(u\hookrightarrow v)$ be a segment of $C$ not contained in $R$. By Lemma \ref{MR_lemma}, $\mathcal{M}\subset R$, so each vertex in $(u\hookrightarrow v)$ is either a non-cut vertex, or a cut vertex in $R_1(G)$; in the latter case, the entire pendant path attached to the vertex is also not in $R$ since otherwise $R$ could not be connected. Suppose $(u\hookrightarrow v)$ contains three distinct cut vertices, $p$, $q$, and $r$, lying on $C$ in this counterclockwise order. Every path from a vertex of $R$ to a vertex in $(p\hookrightarrow r)$ passes through $p$ or $r$. However, once $p$ and $r$ are dominated or forced by some forcing chains starting outside $(u\hookrightarrow v)$, each of $p$ and $r$ will have two uncolored neighbors and will not be able to force another vertex. Thus, the vertices in $(p\hookrightarrow r)$ cannot be forced; note that $(p\hookrightarrow r)\neq \emptyset$ since $q\in (p\hookrightarrow r)$. This contradicts $R$ being a power dominating set, so $(u\hookrightarrow v)$ can contain at most two cut vertices. Similarly, if $(u\hookrightarrow v)$ contains two cut vertices which are not adjacent, then the vertices in $(p\hookrightarrow r)$ cannot be forced. Thus, if $(u\hookrightarrow v)$ contains two cut vertices, they must be adjacent.

Now let $(u\hookrightarrow v)$ be any segment of $C$ which contains either no cut vertices, or one cut vertex in $R_1(G)$, or two adjacent cut vertices in $R_1(G)$. 
We claim that the set $S$ obtained by removing $(u\hookrightarrow v)$ and all pendant paths attached to $(u\hookrightarrow v)$ from $V$ is a connected power dominating set of $G$. 
Indeed, deleting the vertices in $(u\hookrightarrow v)$ from $G$ together with all pendant paths attached to $(u\hookrightarrow v)$ does not disconnect $G$, so $S$ is a connected set. Moreover, note that by definition, a segment cannot include all vertices of $C$; thus, the remaining vertices (or vertex) in $C$ outside $(u\hookrightarrow v)$ can initiate two forcing chains which color $(u\hookrightarrow v)$ and the pendant paths attached to $(u\hookrightarrow v)$. Thus, $(u\hookrightarrow v)$ can be excluded from a connected power dominating set of $G$.
\end{proof}

Let $G=(V,E)$ be a connected graph and $C$ be the vertex set of a block of $G$ such that $G[C]$ is a cycle. We will say a segment $(u\hookrightarrow v)$ of $C$ is \emph{feasible} if it can be excluded from a connected power dominating set of $G$. We will denote by $s(C)$ the maximum size of a feasible segment of $C$. More precisely, in view of Lemma \ref{two_ap_lemma}, $s(C)$ can be defined as follows. Let $\{p_1,\ldots,p_k\}$ be the set of cut vertices in $C$ in counterclockwise order. For $j\in \{0,1,2\}$, if $k\leq j$, let $\mathcal{S}_j(C)=\emptyset$; otherwise, define $\mathcal{S}_j(C)$ as follows, with $i$ read modulo $k$:
\begin{eqnarray*}
\mathcal{S}_0(C)&=&\{(p_i\hookrightarrow p_{i+1}):1\leq i\leq k\}\\
\mathcal{S}_1(C)&=&\{(p_i\hookrightarrow p_{i+2}):1\leq i\leq k, p_{i+1}\in R_1(G)\}\\
\mathcal{S}_2(C)&=&\{(p_i\hookrightarrow p_{i+3}):1\leq i\leq k, p_{i+1}\in R_1(G), p_{i+2}\in R_1(G), p_{i+1}\sim p_{i+2}\}.
\end{eqnarray*} 
Let $\mathcal{S}(C)=\mathcal{S}_0(C)\cup \mathcal{S}_1(C)\cup \mathcal{S}_2(C)$ and if $k>0$, define $s(C)=\max_{S\in \mathcal{S}(C)}\{|S|\}$. Note that $\mathcal{S}_0(C)$ contains all maximal (with respect to inclusion) feasible segments which contain no cut vertices, $\mathcal{S}_1(C)$ contains all maximal feasible segments which contain one cut vertex in $R_1$, and $\mathcal{S}_2(C)$ contains all maximal feasible segments which contain two adjacent cut vertices in $R_1$. Thus, by Lemma  \ref{two_ap_lemma}, $\mathcal{S}(C)$ contains a feasible segment of maximum size. See Figure \ref{figure_cycle} for an illustration.

\begin{figure}[ht!]
\begin{center}
\includegraphics[scale=.35]{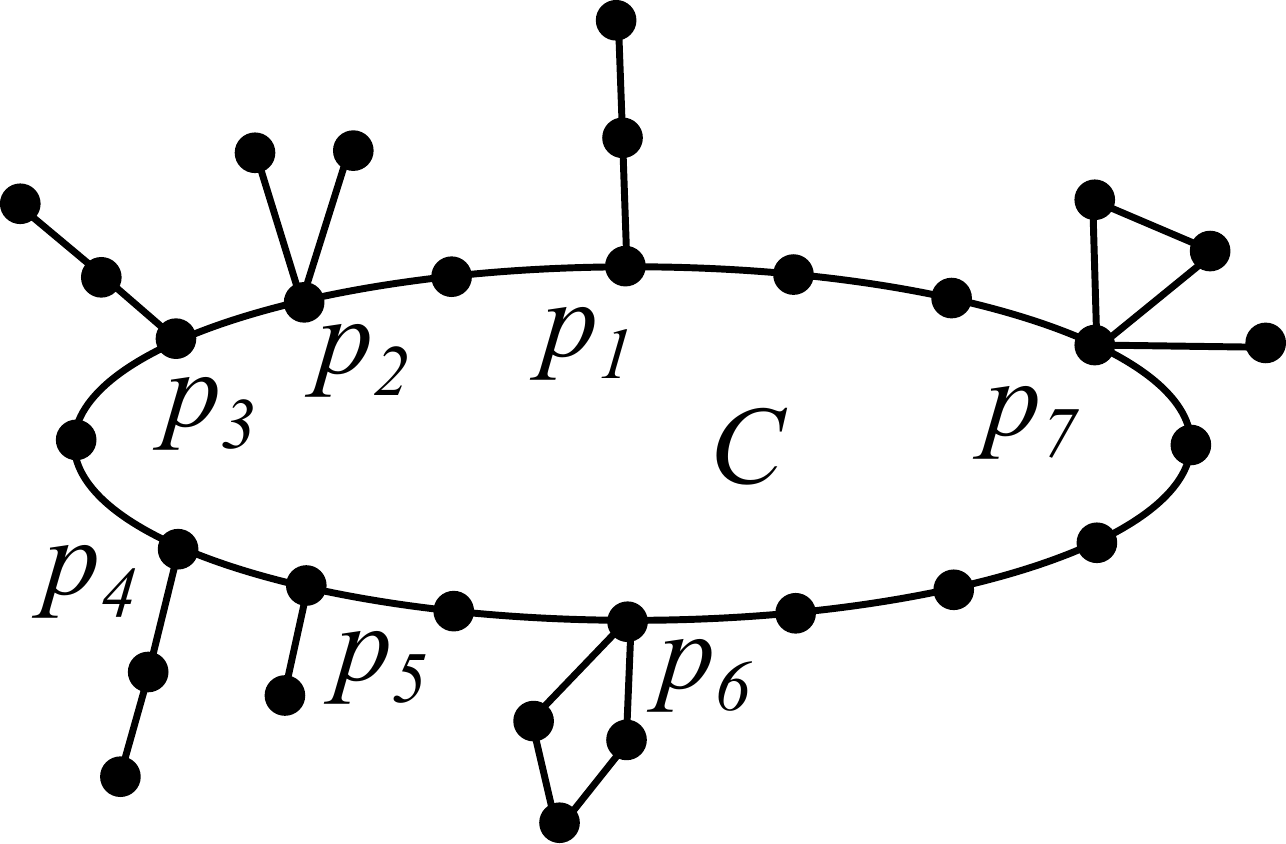}
\caption{Finding maximum size of a feasible segment in a cycle block of a graph: $\mathcal{S}_0(C)=\{(p_1\hookrightarrow p_2),(p_2\hookrightarrow p_3),\ldots, (p_6\hookrightarrow p_7), (p_7\hookrightarrow p_1)\}$; $\mathcal{S}_1(C)=\{(p_2\hookrightarrow p_4), (p_3\hookrightarrow p_5), (p_4\hookrightarrow p_6), (p_7\hookrightarrow p_2)\}$; $\mathcal{S}_2(C)=\{(p_3\hookrightarrow p_6)\}$. Thus, $s(C)=4$.}
\label{figure_cycle}
\end{center}
\end{figure}

\begin{theorem}
\label{thm_cactus}
Let $G=(V,E)$ be a cactus graph, $C_1,\ldots,C_k$ be the vertex sets of the cycles of $G$, and $P_1,\ldots,P_{\ell}$ be the vertex sets of the pendant paths of $G$. Then, 
\begin{equation}
\label{eq_cactus}
\gamma_{P,c}(G)=
\begin{cases}
1 &\text{if } G\simeq C_n \text{ or } G\simeq P_n\\
n-\sum_{i=1}^\ell|P_i|-\sum_{i=1}^k s(C_i) &\text{otherwise}.
\end{cases}
\end{equation} 
Moreover, $\gamma_{P,c}(G)$ can be computed in $O(n)$ time.
\end{theorem}

\proof
If $G\simeq C_n$ or $G\simeq P_n$, clearly $\gamma_{P,c}(G)=1$; thus, assume henceforth that $G\not\simeq C_n$ and $G\not\simeq P_n$. If $G$ is a tree different from a path, by Theorem \ref{tree_thm}, $\gamma_{P,c}(G)=|\mathcal{M}(G)|=n-\sum_{i=1}^{\ell}|P_i|$, and (\ref{eq_cactus}) holds. If $G$ is not a tree nor a cycle, then each cycle of $G$ has at least one cut vertex, so $s(C_i)$ is well-defined for all $i$. For $1\leq i\leq k$, let $S_i$ be a feasible segment of $C_i$ such that $|S_i|=s(C_i)$. We claim that $R^*:=V\backslash ((\bigcup_{i=1}^{\ell} P_i)\cup (\bigcup_{i=1}^{\ell} S_i))$ is a minimum connected power dominating set of $G$.

Clearly, deleting all pendant paths from $G$ does not disconnect it; also, by Lemma \ref{two_ap_lemma}, deleting feasible segments from $G$ also does not disconnect $G$ (given that all pendant paths attached to them are also deleted). Thus, $R^*$ is a connected set. By Lemma \ref{two_ap_lemma}, for $1\leq i\leq k$, $S_i$ and any pendant paths attached to $S_i$ can be forced by the vertices in $C_i\backslash S_i$. Moreover, all other pendant paths are attached to vertices in $R^*$ which can dominate their bases and thus force the entire paths. Thus, $R^*$ is a power dominating set. 

Now suppose there is a minimum connected power dominating set $R'$ of $G$ with $|R'|<|R^*|$. The vertices in $V$ can be partitioned into $\mathcal{M}$, $\bigcup_{i=1}^k C_i\backslash \mathcal{M}$, and $\bigcup_{i=1}^\ell P_i$. By Lemma~\ref{MR_lemma}, $R'$ contains all vertices in $\mathcal{M}$; the rest of $R'$ consists of vertices in $\bigcup_{i=1}^k C_i\backslash \mathcal{M}$, and $\bigcup_{i=1}^\ell P_i$. By Observation \ref{one_seg_obs}, for $1\leq i\leq k$, any vertices of $C_i\backslash \mathcal{M}$ not contained in $R'$ must form a single segment. Suppose first that $R'$ includes some vertices of $\bigcup_{i=1}^\ell P_i$. The pendant paths containing these vertices cannot be attached to vertices of segments excluded from $R'$, since then $R'$ would be disconnected. Thus, these pendant paths are attached to vertices in $R'$. However, the bases of these pendant paths can be dominated by the vertices in $R'$ to which they are attached, and then the entire pendant paths can be forced. Moreover, removing all vertices in pendant paths from $R'$ cannot disconnect the set. Thus, $R'\backslash \bigcup_{i=1}^\ell P_i$ is a smaller connected power dominating set than $R'$, contradicting the minimality of $R'$. It follows that $R'$ does not contain vertices of $\bigcup_{i=1}^\ell P_i$. Thus, for $1\leq i\leq k$, $R'\cap C_i=C_i\backslash (u_i\hookrightarrow v_i)$ for some segment $(u_i\hookrightarrow v_i)$ in $C_i$. However, since $|R'|<|R^*|$, it follows by the pigeonhole principle that there is some $i\in \{1,\ldots,k\}$ for which $|(u_i\hookrightarrow v_i)|<s(C_i)$, a contradiction. Thus, $R^*$ is a minimum connected power dominating set of $G$.

To verify that the time needed to find $\gamma_{P,c}(G)$ is linear in the order of the graph, first note that the set of cut vertices in $G$, and hence the vertices in $\mathcal{M}$, $C_1,\ldots,C_k$, and $P_1,\ldots,P_\ell$ can be found in linear time (cf. \cite{tarjan}). Then, the sets of cut vertices in each cycle of $G$ can also be found in linear time, and each of the sets of segments $\mathcal{S}_0(C_i)$, $\mathcal{S}_1(C_i)$, and $\mathcal{S}_2(C_i)$ can be found in linear time for $1\leq i\leq k$. Since each collection of sets contains $O(n)$ elements, the set of values $s(C_1),\ldots,s(C_k)$ --- and hence $\gamma_{P,c}(G)$ --- can be found in linear time. 
\qed

\section{Integer programming models for $\gamma_P(G)$ and $\gamma_{P,c}(G)$}

In this section, we propose novel integer programming formulations for power domination and connected power domination, and apply them to several power network test cases. The proposed models can also be used to compute the \emph{power propagation time} of $G$, i.e., the minimum number of timesteps required to power dominate a graph by a minimum power dominating set. More precisely, for a graph $G$ and a power dominating set $S\subset V(G)$, define $S^{[1]}=N[S]$ and for $i\geq 1$, define $S^{[i+1]}=S^{[i]}\cup \{w:\exists v\in S^{[i]}, N(v)\cap (V(G)\backslash S^{[i]})=\{w\}\}$. Then, the \emph{power propagation time} of $S$, denoted $\text{ppt}(G, S)$, is the smallest integer $\ell$ such that $S^{[\ell]}=V(G)$. The \emph{power propagation time} of $G$ is defined as $\text{ppt}(G) = \min\{\text{ppt}(G,S):S \text{ is a minimum power dominating set of } G\}$. A related concept is the $\ell$\emph{-round power domination number}, defined to be the minimum number of vertices needed to power dominate $G$ in power propagation time at most $\ell$. The $\ell$-round power domination number has been studied in \cite{aazami,aazami2,ferrero17,liao16}; it is NP-hard to compute even on planar graphs. Our proposed integer programming model can also be used to compute the $\ell$-round power domination number of a graph.

\subsection{Model formulation}

Given a graph $G=(V,E)$, we first transform $G$ into a directed graph $\vec{G}$ over the same vertex set, with edges $(u,v)$ and $(v,u)$ in $\vec{G}$ for each edge $\{u,v\}$ of $G$. For a directed edge $e$ and vertex $v$, we use the the notation $e\in \delta^- (v)$ to indicate that $v$ is the head of $e$. For each $v \in V$, let $s_v \in \{0,1\}$ be a decision variable such that $s_v = 1$ if $v$ is selected to be in the power dominating set, and $s_v=0$ otherwise. Let $T$ be the maximum number of propagation timesteps allowed; note that any graph can be colored in at most $n$ timesteps by any power dominating set. For each $v\in V$, let $x_v \in \{0,1,\ldots,T\}$ be an integer variable that indicates the timestep in which $v$ becomes colored. Lastly, for each directed edge $e$, let $y_e=1$ if the tail of $e$ power dominates the head of $e$, and $y_e=0$ otherwise. With this notation, the power domination problem can be formulated as in Model 1.\\

\noindent \text{\textbf{Model 1: Power domination}}
{\small
\begin{align}
\nonumber\min \;&\sum_{v\in V} s_v\\
\text{s.t. } \;& s_v + \sum_{e\in \delta^- (v)} y_e = 1 \quad & \forall v\in V \label{cons:1} \\
& x_u - x_v + (T+1)y_e \leq T & \forall e= (u,v) \in E \label{cons:2}\\
& x_w - x_v +(T+1)y_e \leq T + (T+1)s_u & \forall e = (u,v), \forall w\in N(u)\setminus \{v\} \label{cons:3}\\
\nonumber & x \in \{0,\ldots,T\}^n,  y \in \{0,1\}^m, s \in \{0,1\}^n.
\end{align}
}

\begin{theorem}
\label{thm_IP}
For any graph $G$, when $T=n$, the optimal value of Model 1 equals $\gamma_P(G)$.
\end{theorem}
\begin{proof}
Let $S$ be a power dominating set of $G=(V,E)$. Then, for each $v\in V$, either $v$ is in $S$ and thus $s_v = 1$, or $v$ is colored by some other vertex $u$ of $G$, i.e., $v$ the head of some edge $uv$ with $y_{uv}=1$. Hence, constraint \eqref{cons:1} is satisfied. Let $x_v$ be the timestep at which vertex $v$ gets colored. Since $T$ is the maximum difference between the timesteps in which any two vertices are colored, for any edge $e$ with $y_e=0$, constraints \eqref{cons:2} and \eqref{cons:3} are satisfied. Moreover, note that a vertex which is not in $S$ cannot color another vertex until all-but-one of its neighbors are colored. Thus, for any edge $e = (u,v)$ such that $y_e = 1$, $u$ must be colored before $v$, and therefore $x_u < x_v$. Additionally, $x_w < x_v$ for all neighbors $w$ of $u$, unless $u\in S$ (if $u\in S$, all neighbors of $u$ get colored, regardless of how many neighbors there are). Therefore, constraints \eqref{cons:2} and \eqref{cons:3} are satisfied, so all constraints are valid for an arbitrary power dominating set $S$.

Now, let $(x,y,s)$ be a feasible solution to Model 1, and let $S = \{v : s_v = 1, v\in V\}$. Let $\mathcal{F}$ be the set of paths formed by the edges for which $y_e = 1$. For an edge $e =(u,v)$ such that $y_e = 1$ and $s_u=0$, by constraints \eqref{cons:2} and \eqref{cons:3} there must be some integer $x_v\in \{0,\ldots,T\}$ such that $x_w+1\leq x_v$ for $w\in N(u)\backslash\{v\}$. By interpreting $x_w$ as the timestep in which vertex $w$ is forced, it follows that there must exist some timestep $x_w\in \{0,\ldots,T\}$ such that $u$ and all its neighbors except $v$ have been forced in previous timesteps. Thus, $u$ can force $v$ in timestep $x_w$. Thus, the paths formed by the edges with $y_e=1$ are forcing chains of $G$. Additionally, by constraint \eqref{cons:1} every vertex $v$ is either in $S$, or has an edge $e = (u,v)$ such that $y_e = 1$, and thus must get colored at some timestep. Therefore, $S$ is a power dominating set of $G$ and $\mathcal{F}$ is a set of forcing chains associated with $S$. Additionally, since $|S|$ is minimized by the objective of Model 1, $S$ is a minimum power dominating set.
\end{proof}


\noindent The following two corollaries easily follow from the proof of Theorem \ref{thm_IP}.

\begin{corollary}
For any graph $G$, when $T=\ell$, the optimal value of Model 1 equals the $\ell$-round power domination  number.
\end{corollary}

\begin{corollary}
For any graph $G$, $\emph{ppt}(G)$ can be found by re-running Model 1 $O(\log n)$ times, using binary search to find the smallest value of $T$ for which the output is the same as for $T=n$.
\end{corollary}

Another integer programming formulation for power domination and $\ell$-round power domination is given by Aazami \cite{aazami2}. Fan and Watson \cite{FanWatson} have also explored integer programming formulations for power domination, as well as connected power domination. However, their formulations model a slightly different problem than the one typically considered in the power domination literature. In particular, the model proposed in \cite{FanWatson}, shown in Model 2, has decision variables $x_i \in \{0,1\}$, $i \in V$, for containment in a power dominating set, and $p_{ij} \in \{0,1\}$, $i,j \in V$, for whether vertex $i$ can power dominate vertex $j$. Further, $a_{i,j}$ is the $i,j^{th}$ element of the adjacency matrix of the graph, and $Z_i = 1$ for $i \in V$ if $i \in V_Z$, where $V_Z$ is a set of zero-injection buses. Zero-injection buses are a set of nodes that can perform the zero forcing propagation in the graph. Therefore, in the model of Fan and Watson, the set of zero-injection buses is assumed to be known \emph{a priori}. If the set $V_Z$ is not known, the model becomes nonlinear, and if it is assumed that all vertices are in $V_Z$, the optimal solution to the model is $x=0$. \\

\noindent \text{\textbf{Model 2: Power domination (Fan and Watson)}}
\begin{align*}
\min \;& \sum_{i} x_i \\
\text{s.t. } \;& \sum_{j} a_{ij} x_j + \sum_{j} a_{ij} Z_i p_{ji} \geq 1 &\forall i \in V \\
& \sum_j a_{ij} p_{ij} = Z_i &\forall i \in V \\
& p_{ij} = 0 &\forall i,j \text{ with } a_{ij} =0 \text{ or } i \notin V_Z \\
& x\in \{0,1\}^n, p \in \{0,1\}^{n^2}
\end{align*}

In order to solve the connected power dominating set problem, additional constraints must be added to ensure connectivity of the power dominating set. There are multiple methods for ensuring the connectivity of a set of vertices, such as the standard minimum spanning tree with sub-tour elimination, Miller-Tucker-Zemlin (MTZ) constraints \cite{miller1960}, Martin constraints \cite{martinCons}, single-commodity flow constraints \cite{commodityCons}, and multi-commodity flow constraints \cite{commodityCons}. Fan and Watson \cite{FanWatson} explored the effectiveness of various connectivity constraints applied to their formulation of power domination, and found that MTZ constraints offered the best computational results. Thus, we apply the MTZ constraints to Model 1 in order to solve the connected power domination problem. 
We used an implementation of these constraints following the method proposed in \cite{quintao2010}, with a modification as explained in \cite{desrochers1991}.

\subsection{Computational Results}

We computed the power domination numbers, connected power domination numbers, and power propagation times of six power graphs from a standard electrical network benchmark dataset  \cite{IEEEtest2012}. Multiple edges and loops were removed from the graph instances. All integer programming formulations were implemented in Julia 0.6.0 using Gurobi 7.5.2; experiments were run on a 2014 MacBook Pro with a 2.6 GHz Intel Core i5, and 8 GB of 1600 MHz DDR3 RAM. The optimality gap in Gurobi was left as default, with a time limit of 7200 seconds.

Computational results are summarized in Tables \ref{fig:compresults} and \ref{fig:compresults2}. Table \ref{fig:compresults} lists the order, size, power domination number, connected power domination number, and the associated runtimes for each graph in the dataset. $T$ is taken to be equal to $|V(G)|$ for each graph $G$, in order to ensure the feasibility of the integer program. As can be seen from Table \ref{fig:compresults}, the connected power domination number is greater than the power domination for the six test graphs. Additionally, for a given $T$, it is generally faster to compute the power domination number than the connected power domination number. Similar results have been observed in other related problems. For example, while both domination and connected domination are NP-complete \cite{GJ}, the latter is generally harder to solve exactly. This disparity has been attributed to the non-locality of the connected domination problem, since exact algorithms are often unable to capture global properties like connectivity \cite{connected_dom}. In some contrast, computational experiments in \cite{brimkov_fast} have shown that algorithms for connected zero forcing are slightly faster than algorithms for zero forcing. Thus, in this aspect, power domination seems to behave more like domination than like zero forcing. In most test cases, the runtimes obtained using Model 1 are lower than the runtimes reported in \cite{FanWatson} obtained using Model 2. Computational results for the integer programming model for power domination proposed in \cite{aazami2} are not available for comparison.

\begin{table}[ht!]
\centering
\label{my-label}
\begin{tabular}{lcc||cc||cc}
             &       &       & \multicolumn{2}{c||}{Power Domination} & \multicolumn{2}{c}{Connected Power Domination} \\ \hline
$G$            & $|V|$ & $|E|$ & $\gamma_P(G)$       & Time (sec)      & $\gamma_{c,P}(G)$         & Time (sec)         \\ \hline
IEEE Bus 14  & 14    & 20    & 2                   & 0.0667          & 2                         & 0.1408             \\
IEEE Bus 30  & 30    & 41    & 3                   & 0.2060          & 4                         & 0.6452             \\
IEEE Bus 57  & 57    & 80    & 3                   & 1.8858          & 6                         & 12.3784            \\
RTS-96       & 73    & 108   & 6                   & 6.4016          & 16                        & 684.5847           \\
IEEE Bus 118 & 118   & 186   & 8                   & 14.9783         & 19                        & 701.5809           \\
IEEE Bus 300 & 300   & 409   & 30                  & 271.0804        & 64*                       & $>7200$               
\end{tabular}
\caption{Computational results for finding minimum power dominating sets using Model 1, and minimum connected power dominating sets using Model 1 with MTZ constraints. Each run was completed with $T = |V(G)|$, with a time limit of 7200 seconds. For IEEE Bus 300, the reported $\gamma_{c,P}(G)$ is the best feasible solution at the point of timeout. }
\label{fig:compresults}
\end{table}

Note that constraints \eqref{cons:2} and \eqref{cons:3} in Model 1 are disjunctive constraints of \emph{big-M} form. Thus, in general, the model is expected to perform better if a there is a small upper bound on the number of steps required to power dominate $G$. We explored this further by re-running the models for different values of $T$ ranging between $n$ and $\text{ppt}(G)$. Table \ref{fig:compresults2} lists the power propagation time $T_{\emph{p}}=\text{ppt}(G)$, the value of $T$ associated with the minimum runtime, denoted $T_{\emph{b}}$, the runtimes associated with these values of $T$, and the average runtime over all $T$, for power domination and connected power domination. From Table \ref{fig:compresults2}, it can be seen that although the best runtimes are not achieved by $T_p$, they are usually achieved by relatively small values of $T$ (relative to the order of the graph), as expected from the \emph{big-M} constraints.

\begin{table}[ht!]
\centering
\begin{tabular}{l||cc|cc|c||cc|cc|c}

             &       \multicolumn{5}{c||}{Power Domination} & \multicolumn{5}{c}{Connected Power Domination} \\ \hline
$G$          & $T_{\emph{p}}$ & Time & $T_{\emph{b}}$ & Time & Avg. & $T_{\emph{p}}$ & Time & $T_{\emph{b}}$ & Time & Avg. \\ \hline
IEEE Bus 14  & 2                 & 0.09       & 6                 & 0.039      & 0.06           & 4                 & 0.18       & 7                 & 0.13       & 0.17\\
IEEE Bus 30  & 6                 & 0.17       & 14                & 0.09       & 0.16           & 6                 & 0.66       & 15                & 0.22       & 0.47\\
IEEE Bus 57  & 21                & 1.38       & 28                & 0.53       & 1.82           & 22                & 16.85      & 27                & 4.80       & 9.23\\
RTS-96       & 7                 & 21.11      & 13                & 2.24       & 5.96           & 7                 & 313.56     & 13                & 137.10     & 321.16\\
IEEE Bus 118 & 16                & 8.42       & 35                & 3.33       & 6.37           & 9                 & 318.91     & 39                & 297.73     & 449.24\\
IEEE Bus 300 & 10                & 197.31     & 17                & 23.96      & 382.06         & -                 & -          & -                 & -          & -              
\end{tabular}
\caption{Computational results for finding power propagation time and runtimes for different values of $T$. Runtimes are reported in seconds. Results for connected power domination of IEEE Bus 300 are unavailable, since they exceeded the timeout limit.}
\label{fig:compresults2}
\end{table}

\section{Conclusion}

In this paper, we presented several structural, algorithmic, and computational results on connected power domination. We explored properties of vertices which are contained in every connected power dominating set, and the effects of certain vertex and edge operations on the connected power domination number. We also gave a formula for computing the connected power domination number of a graph in terms of the connected power domination numbers of its biconnected components. We established the NP-completeness of connected power domination, but showed efficient algorithms for block graphs and cactus graphs. Finally, we gave integer programming models for computing the power domination number, connected power domination number, and power propagation time of a graph, and reported computational results. 

One direction for future work could focus on refining Theorem \ref{np_theorem} and generalizing Theorems \ref{block_graph_thm} and \ref{thm_cactus}. For example, the power domination problem is NP-complete even for bipartite graphs, chordal graphs, planar graphs, and split graphs \cite{guo2005,powerdom3,liao}; on the other hand, efficient algorithms are available for interval graphs \cite{liao}, graphs of bounded treewidth \cite{aazami2009,kneis}, and several other families. It would be interesting to determine whether connected power domination is NP-complete or polynomially-solvable for these classes of graphs. Another problem of interest is to generalize Theorem \ref{theorem_blocks} to separating sets of larger size. Such theoretical refinements could also be leveraged in general-purpose solution approaches, such as the integer programming models given in Section 6.

\section*{Acknowledgements}
This work is supported by the National Science Foundation, under Grants CMMI-1300477 and CMMI-1404864.

\bibliographystyle{abbrv}
\bibliography{mybib}

\begin{thebibliography}{10}

\bibitem{IEEEtest2012}
{IEEE} reliability test data, 2012.
\newblock \url{http://www.ee.washington.edu/research/pstca/}.

\bibitem{aazami}
A.~Aazami.
\newblock {\em Hardness results and approximation algorithms for some problems
  on graphs}.
\newblock PhD thesis, PhD thesis. University of Waterloo, 2008.

\bibitem{aazami2}
A.~Aazami.
\newblock Domination in graphs with bounded propagation: algorithms,
  formulations and hardness results.
\newblock {\em Journal of Combinatorial Optimization}, 19(4):429--456, 2010.

\bibitem{aazami2009}
A.~Aazami and K.~Stilp.
\newblock Approximation algorithms and hardness for domination with
  propagation.
\newblock {\em SIAM Journal on Discrete Mathematics}, 23(3):1382--1399, 2009.

\bibitem{AIM-Workshop}
{AIM Special Work Group}.
\newblock Zero forcing sets and the minimum rank of graphs.
\newblock {\em Linear Algebra and its Applications}, 428(7):1628--1648, 2008.

\bibitem{EEprobabilistic}
F.~Aminifar, M.~Fotuhi-Firuzabad, M.~Shahidehpour, and A.~Khodaei.
\newblock Probabilistic multistage {PMU} placement in electric power systems.
\newblock {\em IEEE Transactions on Power Delivery}, 26(2):841--849, 2011.

\bibitem{Baldwin93}
T.~Baldwin, L.~Mili, M.~Boisen, and R.~Adapa.
\newblock Power system observability with minimal phasor measurement placement.
\newblock {\em IEEE Transactions on Power Systems}, 8(2):707--715, 1993.

\bibitem{barioli_minrank}
F.~Barioli, S.~Fallat, and L.~Hogben.
\newblock Computation of minimal rank and path cover number for certain graphs.
\newblock {\em Linear Algebra and its Applications}, 392:289--303, 2004.

\bibitem{barioli_pathcover}
F.~Barioli, S.~Fallat, and L.~Hogben.
\newblock On the difference between the maximum multiplicity and path cover
  number for tree-like graphs.
\newblock {\em Linear Algebra and its Applications}, 409:13--31, 2005.

\bibitem{benson}
K.~F. Benson, D.~Ferrero, M.~Flagg, V.~Furst, L.~Hogben, V.~Vasilevska, and
  B.~Wissman.
\newblock Power domination and zero forcing.
\newblock {\em arXiv:1510.02421}, 2015.

\bibitem{bondy}
J.~A. Bondy and U.~S.~R. Murty.
\newblock {\em Graph Theory with Applications}, volume 290.
\newblock Macmillan London, 1976.

\bibitem{restrictedPD}
C.~Bozeman, B.~Brimkov, C.~Erickson, D.~Ferrero, M.~Flagg, and L.~Hogben.
\newblock Restricted power domination and zero forcing problems.
\newblock {\em arXiv:1711.05190}, 2017.

\bibitem{brimkov_fast}
B.~Brimkov, C.~C. Fast, and I.~V. Hicks.
\newblock Computational approaches for zero forcing and related problems.
\newblock {\em arXiv:1704.02065}, 2017.

\bibitem{brimkov_extremal}
B.~Brimkov, C.~C. Fast, and I.~V. Hicks.
\newblock Graphs with extremal connected forcing numbers.
\newblock {\em arXiv:1701.08500}, 2017.

\bibitem{brimkov_comp_and_complex}
B.~Brimkov and I.~V. Hicks.
\newblock Complexity and computation of connected zero forcing.
\newblock {\em Discrete Applied Mathematics}, 229:31 -- 45, 2017.

\bibitem{Brunei93}
D.~J. Brueni.
\newblock {\em Minimal PMU placement for graph observability: a decomposition
  approach}.
\newblock PhD thesis, Virginia Tech, 1993.

\bibitem{BH05}
D.~J. Brueni and L.~S. Heath.
\newblock The {PMU} placement problem.
\newblock {\em SIAM Journal on Discrete Mathematics}, 19(3):744--761, 2005.

\bibitem{quantum1}
D.~Burgarth and V.~Giovannetti.
\newblock Full control by locally induced relaxation.
\newblock {\em Physical Review Letters}, 99(10):100501, 2007.

\bibitem{con_vc2}
E.~Camby, J.~Cardinal, S.~Fiorini, and O.~Schaudt.
\newblock The price of connectivity for vertex cover.
\newblock {\em Discrete Mathematics and Theoretical Computer Science},
  16:207--223, 2014.

\bibitem{Caro}
Y.~Caro, D.~B. West, and R.~Yuster.
\newblock Connected domination and spanning trees with many leaves.
\newblock {\em SIAM Journal on Discrete Mathematics}, 13(2):202--211, 2000.

\bibitem{chang2012}
G.~J. Chang, P.~Dorbec, M.~Montassier, and A.~Raspaud.
\newblock Generalized power domination of graphs.
\newblock {\em Discrete Applied Mathematics}, 160(12):1691 -- 1698, 2012.

\bibitem{chang2015}
G.~J. Chang and N.~Roussel.
\newblock On the $k$-power domination of hypergraphs.
\newblock {\em Journal of Combinatorial Optimization}, 30(4):1095--1106, 2015.

\bibitem{Desormeaux}
W.~J. Desormeaux, T.~W. Haynes, and M.~A. Henning.
\newblock Bounds on the connected domination number of a graph.
\newblock {\em Discrete Applied Mathematics}, 161(18):2925--2931, 2013.

\bibitem{desrochers1991}
M.~Desrochers and G.~Laporte.
\newblock Improvements and extensions to the {M}iller-{T}ucker-{Z}emlin subtour
  elimination constraints.
\newblock {\em Operations Research Letters}, 10(1):27--36, 1991.

\bibitem{commodityCons}
B.~N. Dilkina and C.~P. Gomes.
\newblock Solving connected subgraph problems in wildlife conservation.
\newblock In {\em CPAIOR}, volume 6140, pages 102--116. Springer, 2010.

\bibitem{Edholm}
C.~J. Edholm, L.~Hogben, J.~LaGrange, and D.~D. Row.
\newblock Vertex and edge spread of zero forcing number, maximum nullity, and
  minimum rank of a graph.
\newblock {\em Linear Algebra and its Applications}, 436(12):4352--4372, 2012.

\bibitem{FanWatson}
N.~Fan and J.-P. Watson.
\newblock Solving the connected dominating set problem and power dominating set
  problem by integer programming.
\newblock In {\em International Conference on Combinatorial Optimization and
  Applications}, pages 371--383. Springer, 2012.

\bibitem{ferrero17}
D.~Ferrero, L.~Hogben, F.~H. Kenter, and M.~Young.
\newblock Note on power propagation time and lower bounds for the power
  domination number.
\newblock {\em Journal of Combinatorial Optimization}, 34(3):736--741, 2017.

\bibitem{connected_dom}
F.~V. Fomin, F.~Grandoni, and D.~Kratsch.
\newblock Solving connected dominating set faster than $2^n$.
\newblock {\em Algorithmica}, 52(2):153--166, 2008.

\bibitem{GJ}
M.~R. Garey and D.~S. Johnson.
\newblock {\em Computers and intractability}.
\newblock W.H. Freeman \& Co., San Francisco, 1979.

\bibitem{guo2005}
J.~Guo, R.~Niedermeier, and D.~Raible.
\newblock Improved algorithms and complexity results for power domination in
  graphs.
\newblock In {\em FCT}, volume 3623, pages 172--184. Springer, 2005.

\bibitem{powerdom3}
T.~W. Haynes, S.~M. Hedetniemi, S.~T. Hedetniemi, and M.~A. Henning.
\newblock Domination in graphs applied to electric power networks.
\newblock {\em SIAM Journal on Discrete Mathematics}, 15(4):519--529, 2002.

\bibitem{Huang}
L.-H. Huang, G.~J. Chang, and H.-G. Yeh.
\newblock On minimum rank and zero forcing sets of a graph.
\newblock {\em Linear Algebra and its Applications}, 432(11):2961--2973, 2010.

\bibitem{kneis}
J.~Kneis, D.~M{\"o}lle, S.~Richter, and P.~Rossmanith.
\newblock Parameterized power domination complexity.
\newblock {\em Information Processing Letters}, 98(4):145--149, 2006.

\bibitem{EEinformationTheoretic}
Q.~Li, T.~Cui, Y.~Weng, R.~Negi, F.~Franchetti, and M.~D. Ilic.
\newblock An information-theoretic approach to {PMU} placement in electric
  power systems.
\newblock {\em IEEE Transactions on Smart Grid}, 4(1):446--456, 2013.

\bibitem{con_vc1}
Y.~Li, Z.~Yang, and W.~Wang.
\newblock Complexity and algorithms for the connected vertex cover problem in
  4-regular graphs.
\newblock {\em Applied Mathematics and Computation}, 301:107--114, 2017.

\bibitem{liao16}
C.-S. Liao.
\newblock Power domination with bounded time constraints.
\newblock {\em Journal of Combinatorial Optimization}, 31(2):725--742, 2016.

\bibitem{liao}
C.-S. Liao and D.-T. Lee.
\newblock Power domination problem in graphs.
\newblock In {\em COCOON}, pages 818--828. Springer, 2005.

\bibitem{EEtaxonomy}
N.~M. Manousakis, G.~N. Korres, and P.~S. Georgilakis.
\newblock Taxonomy of {PMU} placement methodologies.
\newblock {\em IEEE Transactions on Power Systems}, 27(2):1070--1077, 2012.

\bibitem{martinCons}
R.~K. Martin.
\newblock Using separation algorithms to generate mixed integer model
  reformulations.
\newblock {\em Operations Research Letters}, 10(3):119--128, 1991.

\bibitem{Mili91}
L.~Mili, T.~Baldwin, and A.~Phadke.
\newblock Phasor measurements for voltage and transient stability monitoring
  and control.
\newblock In {\em Workshop on Application of advanced mathematics to Power
  Systems, San Francisco}, 1991.

\bibitem{miller1960}
C.~E. Miller, A.~W. Tucker, and R.~A. Zemlin.
\newblock Integer programming formulation of traveling salesman problems.
\newblock {\em Journal of the ACM}, 7(4):326--329, 1960.

\bibitem{nylen}
P.~M. Nylen.
\newblock Minimum-rank matrices with prescribed graph.
\newblock {\em Linear Algebra and its Applications}, 248:303--316, 1996.

\bibitem{EEtabuSearch}
J.~Peng, Y.~Sun, and H.~Wang.
\newblock Optimal {PMU} placement for full network observability using {T}abu
  search algorithm.
\newblock {\em International Journal of Electrical Power \& Energy Systems},
  28(4):223--231, 2006.

\bibitem{quintao2010}
F.~P. Quint{\~a}o, A.~S. da~Cunha, G.~R. Mateus, and A.~Lucena.
\newblock The k-cardinality tree problem: reformulations and lagrangian
  relaxation.
\newblock {\em Discrete Applied Mathematics}, 158(12):1305--1314, 2010.

\bibitem{row_cacti}
D.~Row.
\newblock Zero forcing number, path cover number, and maximum nullity of cacti.
\newblock {\em Involve, a Journal of Mathematics}, 4(3):277--291, 2012.

\bibitem{Sampathkumar}
E.~Sampathkumar and H.~B. Walikar.
\newblock The connected domination number of a graph.
\newblock {\em Journal of Mathematical and Physical Sciences}, 13(6):607--613,
  1979.

\bibitem{EEmultiStage}
R.~Sodhi, S.~Srivastava, and S.~Singh.
\newblock Multi-criteria decision-making approach for multi-stage optimal
  placement of phasor measurement units.
\newblock {\em IET Generation, Transmission \& Distribution}, 5(2):181--190,
  2011.

\bibitem{tarjan}
R.~E. Tarjan.
\newblock A note on finding the bridges of a graph.
\newblock {\em Information Processing Letters}, 2(6):160--161, 1974.

\bibitem{fast_mixed_search}
B.~Yang.
\newblock Fast--mixed searching and related problems on graphs.
\newblock {\em Theoretical Computer Science}, 507:100--113, 2013.

\end{thebibliography}

\end{document}